\documentclass[a4paper,12pt]{amsart}

\usepackage{amsmath}
\usepackage{amssymb}
\usepackage{amsfonts}

\newtheorem{thm}{Theorem}[section]
\newtheorem{prop}[thm]{Proposition}
\newtheorem{lem}[thm]{Lemma}
\newtheorem{cor}[thm]{Corollary}
\newtheorem{rem}[thm]{Remark}
\newtheorem{rems}[thm]{Remarks}
\newtheorem{defi}[thm]{Definition}
\newtheorem{exo}{\bf\large Exercice}

\newcommand{\R}{\mathbb{R}}

\newcommand{\N}{\mathbb{N}}
\newcommand{\C}{\mathbb{C}}

\newcommand{\beq}{\begin{eqnarray}}
\newcommand{\eeq}{\end{eqnarray}}
\newcommand{\bq}{\begin{equation}}
\newcommand{\eq}{\end{equation}}
\newcommand{\beqn}{\begin{eqnarray*}}
\newcommand{\eeqn}{\end{eqnarray*}}
\newcommand{\bex}{\begin{exo}}
\newcommand{\eex}{\end{exo}}
\newcommand{\ben}{\begin{enumerate}}
\newcommand{\een}{\end{enumerate}}

\newcommand{\Int}{\displaystyle \int}
\newcommand{\Frac}{\displaystyle \frac}

\newcommand{\Sup}{\displaystyle \sup}
\newcommand{\Lim}{\displaystyle \lim}
\newcommand{\Liminf}{\displaystyle \liminf}
\newcommand{\Limsup}{\displaystyle \limsup}
\newcommand{\Max}{\displaystyle \max}

\makeatletter

\@addtoreset{equation}{section}
\makeatother

\author{Ines Ben Ayed}
\address{Universit\'e de Tunis El Manar,
Facult\'e des Sciences de Tunis, D\'epartement de Math\'ematiquess, 2092, Tunis, Tunisie}
\email{\sl abenyed08@gmail.com}

\author{Mohamed Khalil Zghal}
\address{Universit\'e de Tunis El Manar,
Facult\'e des Sciences de Tunis, D\'epartement de Math\'ematiquess, 2092, Tunis, Tunisie}
\email{\sl zghal-khalil@hotmail.fr}
\thanks{{\sf I. Ben Ayed \& M.- K. Zghal} are grateful to the Laboratory of
PDE and Applications at the Faculty of Sciences of Tunis}

\title[Characterization of the lack of...]
{Characterization of the lack of compactness of $H^2_{rad}(\R^4)$ into the Orlicz space}
\begin{document}
\begin{abstract}
This paper is devoted to the description of the lack of
compactness of the Sobolev space $H^2_{rad}(\R^4)$ in the Orlicz space $\mathcal{L}(\R^4)$. The approach that we adopt to establish this characterization is in the spirit of the one adopted in
the case of $H^1_{rad}(\R^2)$ into the Orlicz space $\mathcal{L}(\R^2)$ in \cite{Bahouri}.
\end{abstract}
\maketitle \tableofcontents

\section{Introduction}
\subsection{Development in critical Sobolev embedding}
Due to the scaling invariance, the critical Sobolev embedding
\begin{equation}\label{inj}
    \dot{H}^s(\mathbb{R}^d)\hookrightarrow L^p(\mathbb{R}^d),
\end{equation}
 when $0\leq s< \frac{d}{2}$ and $\frac{1}{p}=\frac{1}{2}-\frac{s}{d}$, is not compact.\\
 After the pioneering works of P. Lions \cite{Lions1} and \cite{Lions2}, P. G\'{e}rard described in \cite{Gerard} the lack
 of compactness of \eqref{inj} by means of profiles in the following terms: a sequence $(u_n)_n$ bounded
 in $\dot{H}^s(\R^d)$ can be decomposed, up to a subsequence extraction, on a finite sum of orthogonal
 profiles such that the remainder converges to zero in $L^p(\R^d)$ as the number of the sum and $n$ tend to infinity.
 This question was later investigated by S. Jaffard in the more general case of $H^{s,q}(\mathbb{R}^d)\hookrightarrow L^p(\mathbb{R}^d)$,
 $0<s<\frac{d}{p}$ and $\frac{1}{p}=\frac{1}{q}-\frac{s}{d}$ by the use of nonlinear wavelet and recently in an abstract frame $X\hookrightarrow Y$ including Sobolev, Besov, Triebel-Lizorkin, Lorentz, H\"older and BMO spaces. (One can consult ~\cite{BCD} and the references therein for an introduction to these  spaces). In addition, in \cite{Bahouri1}, \cite{Bahouri2} and \cite{Bahouri} H. Bahouri, M. Majdoub and N. Masmoudi characterized the lack of compactness of $H^1(\mathbb{R}^2)$ in the Orlicz space (see Definition \ref{deforl})
 \begin{equation*}
    H^1(\mathbb{R}^2)\hookrightarrow \mathcal{L}(\R^2),
 \end{equation*}
in terms of orthogonal profiles generalizing the example by Moser:
$$ g_n(x):=\sqrt{\frac{\alpha_n}{2\pi}}\;\psi\Big(\frac{-\log|x|}{\alpha_n}\Big), $$
where  $\underline{\alpha}:=(\alpha_n)$, called the scale, is a sequence of positive real numbers going to infinity and ~$\psi $, called the profile, belongs to the set
$$\Big\{\;\psi\in L^2(\R,{\rm e}^{-2s}ds);\;\;\; \psi'\in
L^2(\R),\;\psi_{|]-\infty,0]}=0\,\Big\}.
$$
The study of the lack of compactness of critical Sobolev embedding was at the origin of several works concerning the understanding of features of solutions of  nonlinear partial differential equations. Among others, one can mention \cite{BG}, \cite{km}, \cite{ker}, \cite{La} and \cite{Tao}.
\subsection{Critical 4D Sobolev embedding}
The Sobolev space $H^2(\mathbb{R}^4)$ is continuously embedded in
all Lebesgue spaces $L^p(\mathbb{R}^4)$ for all $2 \leq p < \infty.$
On the other hand, it is also known that $H^2(\mathbb{R}^4)$ embed
in $BMO(\mathbb{R}^4)\cap L^2(\mathbb{R}^4)$, where
~$BMO(\mathbb{R}^d)$ denotes the space of bounded mean oscillations
which is  the space of locally integrable functions $f$ such that
$$ \|f\|_{BMO}=\sup_{B}\frac 1 {|B|}\int_B |f-f_B|\;dx< \infty \quad \mbox{with}\quad f_B= \frac 1 {|B|}\int_B f \;dx.$$
The above supremum being taken over the set of Euclidean balls~$B$,~$|\cdot |$ denoting the Lebesgue measure.\\\\
In this paper, our goal is to investigate the lack of compactness of the Sobolev space $H^2_{rad}(\R^4)$ in the Orlicz space $\mathcal{L}(\R^4)$ defined as follows:
\begin{defi}\label{deforl} Let $\phi : \R^+\to\R^+$ be a convex increasing function such that
$$
\phi(0)=0=\lim_{s\to 0^+}\,\phi(s),\quad
\lim_{s\to\infty}\,\phi(s)=\infty.
$$
We say that a measurable function $u : \R^d\to\C$ belongs to
$L^\phi$ if there exists $\lambda>0$ such that
$$
\int_{\R^d}\,\phi\left(\frac{|u(x)|}{\lambda}\right)\;dx<\infty.
$$
We denote then
\begin{equation*}
\|u\|_{L^\phi}=\inf\,\left\{\,\lambda>0,\int_{\R^d}\,\phi\left(\frac{|u(x)|}{\lambda}\right)\;dx\leq
1\,\right\}.
\end{equation*}
\end{defi}
\noindent In what follows we shall fix  $d=4$, $\phi(s)={\rm e}^{ s^2}-1$ and
denote the Orlicz space $L^\phi$ by ${\mathcal L}$ endowed with the
norm $\|\cdot\|_{\mathcal L}$ where the number $1$ is replaced by
the constant $\kappa$ involved in \eqref{2}. It is easy to see that $\mathcal{L} \hookrightarrow L^p$ for every $2 \leq p <\infty$.\\
The 4D Sobolev embedding in Orlicz space $\mathcal{L}$ states as follows:
\begin{equation}\label{inje}
    \|u\|_{\mathcal{L}(\R^4)}\leq\frac{1}{\sqrt{32\pi^2}}\|u\|_{H^2(\R^4)}.
\end{equation}
Inequality \eqref{inje} derives immediately from the following
proposition due to Ruf and Sani in \cite{sharp}:
\begin{prop}\label{propr} There exists a finite constant $\kappa>0$ such that
\begin{equation}\label{2}
\sup_{u\in H^2(\R^4),\|u\|_{H^2(\R^4)}\leq
1}\;\;\int_{\R^4}\,\left({\rm e}^{32\pi^2
|u(x)|^2}-1\right)\;dx:=\kappa.
\end{equation}
\end{prop}
\noindent Let us notice that if we only require that $\|\Delta
u\|_{L^2(\mathbb{R}^4)}\leq 1$ then the following result established in \cite{com} holds.
\begin{prop}
Let $\beta\in[0,32\pi^2[$, then there exists $C_\beta>0$ such that
  \begin{equation}\label{3}
    \int_{\R^4}\left({\rm e}^{\beta |u(x)|^2}-1\right)dx\leq C_\beta\|u\|_{L^2(\R^4)}^2\quad\forall\, u\in H^2(\R^4)\;\mbox{with}\;\|\Delta u\|_{L^2}\leq 1,
  \end{equation}
and this inequality is false for $\beta\geq 32\pi^2$.
\end{prop}
\begin{rems}The well-known following properties can be found in \cite{com} and \cite{sharp}.\\
  {\bf a)} The inequality \eqref{2} is sharp.\\
  {\bf b)} There exists a positive constant $C$ such that for any domain $\Omega\subseteq\R^4$ 
$$\sup_{u\in H^2(\Omega),\|(-\Delta+I)u\|_{L^2(\Omega)}\leq 1}\Int_\Omega \Big({\rm e}^{32\pi^2|u(x)|^2}-1\Big)\; dx\leq C.$$
{\bf c)} In dimension 2, the inequality \eqref{3} is replaced by the following Trudinger-Moser type inequality (see \cite{AT} and \cite{Ruf}):\\
 Let $\alpha\in [0,4\pi [$. A constant $C_\alpha$ exists such that
\begin{equation}
\label{Mos1} \int_{\R^2}\,\left({\rm e}^{\alpha
|u(x)|^2}-1\right)\;dx\leq C_\alpha \|u\|_{L^2(\R^2)}^2\quad\forall\,u \in H^1(\R^2)\;\mbox{with}\;\|\nabla
u\|_{L^2(\R^2)}\leq 1. 
\end{equation}
 Moreover, if $\alpha\geq 4\pi$ then
\eqref{Mos1} is false.
\end{rems}
\subsection{Lack of compactness in 4D critical Sobolev embedding in Orlicz space}
The embedding of $H^2(\R^4)$ into the Orlicz space is non compact. Firstly, we have a lack of compactness at infinity as shown by the following example:
$$u_k(x)=\varphi(x+x_k),\quad \varphi\in \mathcal{D}(\R^4)\setminus\left\{0\right\}\quad \mbox{and}\quad|x_k|\underset{k\rightarrow\infty}\longrightarrow \infty.$$
 Secondly, we have a lack of compactness generated by a concentration phenomenon as illustrated by the following example (see \cite{sharp} for instance):
\begin{equation}\label{exem}
 f_{\alpha}(x)=\left\{%
\begin{array}{ll}
\sqrt{\frac{\alpha}{8 {\pi}^2}}+ \frac{1-|x|^2 \rm{e}^{2\alpha}}{\sqrt{32\pi^2 \alpha}} & \hbox{if $|x|\leq {\rm{e}}^{-\alpha}$} \\\\
\frac{-\log|x|}{\sqrt{8\pi^2 \alpha}} &\hbox{if ${\rm{e}}^{-\alpha}< |x|\leq 1$}\\\\
\eta_{\alpha}(x)&\hbox{if $|x|> 1$},
\end{array}%
\right.
\end{equation}
where $\eta_{\alpha}\in \mathcal{D}(\R^4)$ and satisfies the following boundary conditions:
$$ {\eta_{\alpha}}_{|\partial B_1}=0,\;
  {\frac{\partial\eta_{\alpha}}{\partial \nu}}_{\Big|\partial B_1}=\frac{1}{\sqrt{8\pi^2\alpha}},$$
with $B_1$ is the unit ball in $\R^4$. In addition, $\eta_{\alpha}$, $\nabla \eta_{\alpha}$, $\Delta \eta_{\alpha}$ are all equal to $O\Big(\Frac{1}{\sqrt{\alpha}}\Big)($\footnote{The notation $g(\alpha)=O(h(\alpha))$ as $\alpha\rightarrow +\infty$, where $g$ and $h$ are two functions defined on some neighborhood of infinity, means the existence of positive numbers $\alpha_0$ and $C$ such that for any $\alpha> \alpha_0$ we have $|g(\alpha)|\leq C |h(\alpha)|$.}) as $\alpha$ tends to infinity.\\
By a simple calculation (see Appendix A), we obtain that
$$\|f_{\alpha}\|_{L^2}^2=O\Big(\frac{1}{\alpha}\Big),\; \|\nabla f_{\alpha}\|_{L^2}^2=O\Big(\frac{1}{\alpha}\Big)\;\mbox{and}\; \|\Delta f_{\alpha}\|_{L^2}^2=1+O\Big(\frac{1}{\alpha}\Big)\quad\mbox{as}\;\alpha\rightarrow+\infty.$$
 Also, we can see that $f_{\alpha}\underset{\alpha\rightarrow\infty}\rightharpoonup 0$ in $H^2(\mathbb{R}^4)$.\\
The lack of compactness in the Orlicz space $\mathcal{L}(\R^4)$ displayed by the sequence $(f_{\alpha})$ when $\alpha$ goes to infinity can be stated qualitatively as follows:
\begin{prop} The sequence $(f_{\alpha})$ defined by \eqref{exem} satisfies:
$$\|f_{\alpha}\|_{\mathcal{L}}\rightarrow \frac{1}{\sqrt{32\pi^2}},\mbox{\,as\,}\, \alpha\rightarrow +\infty.$$
\end{prop}
\begin{proof}
Firstly, we shall prove that $\Liminf_{\alpha\rightarrow \infty}\|f_{\alpha}\|_{\mathcal{L}}\geq\frac{1}{\sqrt{32\pi^2}}.$
For that purpose, let us consider $\lambda > 0$ such that
\begin{equation*}
  \displaystyle\int_{\mathbb{R}^4}\left({\rm e}^{\frac{|f_\alpha(x)|^2}{\lambda^2}}-1\right)\;dx \leq \kappa.
\end{equation*}
Then
\begin{equation*}
    \int_{ |x|\leq \rm{e}^{-\alpha}} \left({\rm e}^{\frac{|f_{\alpha}(x)|^2}{\lambda^2}}-1\right) dx \leq \kappa.
\end{equation*}
But for $|x|\leq \rm{e}^{-\alpha}$, we have
$$f_{\alpha}(x)=\sqrt{\frac{\alpha}{8\pi^2}}+\frac{1-|x|^2 \rm{e}^{2\alpha}}{\sqrt{32\pi^2 \alpha}}\geq \sqrt{\frac{\alpha}{8\pi^2}}.$$
So we deduce that\begin{equation*}
        \displaystyle 2\pi^2 \int_{0}^{{\rm e}^{-\alpha}}\left({\rm e}^{\frac{\alpha}{8\pi^2 \lambda^2}}-1\right)r^3\;dr  \leq \kappa.
       \end{equation*}
Consequently,
 \begin{equation*}
            2\pi^2 \left(\rm{e}^{\frac{\alpha}{8\pi^2\lambda^2}}-1\right)\frac{\rm{e}^{-4\alpha}}{4}\leq \kappa,
            \end{equation*}
which implies that \begin{eqnarray*}
           \lambda^2 &\geq&\Frac{1}{ 32\pi^2+\frac{8\pi^2}{\alpha} \log(\frac{2\kappa}{\pi^2}+\rm{e}^{-4\alpha})}\underset{\alpha\rightarrow \infty}{\longrightarrow}\frac{1}{32\pi^2}.
         \end{eqnarray*}
This ensures that
          $$\Liminf_{\alpha\rightarrow \infty}\|f_\alpha\|_{\mathcal{L}} \geq \frac{1}{\sqrt{32\pi^2}}.$$
To conclude, it suffices to show that $\Limsup_{\alpha\rightarrow \infty}\|f_\alpha\|_{\mathcal{L}} \leq \frac{1 }{\sqrt{32\pi^2}}$. To go to this end, let us fix $\varepsilon > 0$ and use Inequality \eqref{3} with $\beta=32\pi^2-\varepsilon$. Thus, there exists $C_{\varepsilon}>0$ such that
\begin{eqnarray*}
  \int_{\mathbb{R}^4} \left({\rm e}^{(32\pi^2-\varepsilon)\frac{|f_{\alpha}(x)|^2}{\|\Delta f_{\alpha}\|_{L^2}^{2}}}-1\right) \;dx
 \leq C_{\varepsilon}\frac{\|f_\alpha\|_{L^2}^2}{\|\Delta f_{\alpha}\|_{L^2}^2}.
\end{eqnarray*}
The fact that $\underset{\alpha \rightarrow \infty}{\lim}\|f_\alpha\|_{L^2}=0$  leads to
\begin{equation*}
  \Limsup_{\alpha\rightarrow \infty}\|f_{\alpha}\|_{\mathcal{L}}^2\leq \frac{1}{32\pi^2-\varepsilon},
\end{equation*}
which ends the proof of the result.
\end{proof}
The following result specifies the concentration effect revealed by the family $(f_\alpha)$:
\begin{prop}\label{concentration}
With the above notation, we have
\begin{eqnarray*}
|\Delta f_{\alpha}|^2\rightarrow \delta(x=0)\quad\mbox{and}\quad {\rm e}^{32\pi^2|f_{\alpha}|^2}-1\rightarrow \frac{\pi^2}{16}({\rm e}^4+3)\delta(x=0)\quad\mbox{as}\quad \alpha\rightarrow\infty\quad\mbox{in}\quad \mathcal{D}'(\R^4).
\end{eqnarray*}
\end{prop}
\begin{proof}
For any smooth compactly supported function $\varphi$, let us write
\begin{eqnarray*}
  \Int_{\R^4} |\Delta f_{\alpha}(x)|^2 \varphi(x)\; dx &=& I_\alpha+J_\alpha+K_\alpha,
\end{eqnarray*}
with
\begin{eqnarray*}
  I_\alpha &=& \int_{|x|\leq {\rm{e}}^{-\alpha}} |\Delta f_{\alpha}(x)|^2 \varphi(x)\; dx,\\
  J_\alpha &=& \int_{{\rm{e}}^{-\alpha}\leq |x|\leq 1} |\Delta f_{\alpha}(x)|^2 \varphi(x) \; dx\quad\mbox{and} \\
  K_\alpha &=& \int_{|x|\geq1} |\Delta f_{\alpha}(x)|^2 \varphi(x)\; dx.
\end{eqnarray*}
Noticing that $\Delta f_{\alpha}(x)=\frac{-8{\rm{e}}^{2\alpha}}{\sqrt{32\pi^2\alpha}}$ if $|x|\leq {\rm{e}}^{-\alpha}$, we get
\begin{equation*}
|I_\alpha|\leq\frac{\|\varphi\|_{L^{\infty}}}{\alpha}\underset{\alpha\rightarrow \infty}{\longrightarrow}0.
\end{equation*}
This ends the proof of the first assertion.

On the other hand, as $\Delta f_{\alpha}=\frac{-2}{|x|^2 \sqrt{8\pi^2 \alpha}}$ if ${\rm{e}}^{-\alpha}\leq |x|\leq 1,$
we get
\begin{eqnarray*}
  J_\alpha
  &=& \frac{1}{2\pi^2 \alpha}\Int_{{\rm{e}}^{-\alpha}\leq |x|\leq 1} \frac{1}{|x|^4} \varphi(0)\; dx + \frac{1}{2\pi^2 \alpha}\Int_{{\rm{e}}^{-\alpha}\leq |x|\leq 1}  \frac{1}{|x|^4} \big(\varphi(x)-\varphi(0)\big)\; dx\\
  &=& \varphi(0)+ \frac{1}{2\pi^2 \alpha}\Int_{{\rm{e}}^{-\alpha}\leq |x|\leq 1}  \frac{1}{|x|^4} \big(\varphi(x)-\varphi(0)\big)\; dx.
\end{eqnarray*}
Using the fact that $|\varphi(x)-\varphi(0)|\leq |x| \|\nabla \varphi\|_{L^{\infty}}$ we obtain that
\begin{eqnarray*}
|J_\alpha-\varphi(0)|\leq \frac{\|\nabla \varphi\|_{L^{\infty}}}{\alpha}(1-{\rm{e}}^{-\alpha})\underset{\alpha\rightarrow \infty}{\longrightarrow} 0.
 \end{eqnarray*}
Finally, taking advantage of the existence of a positive constant $C$ such that\\ $\|\Delta\eta_{\alpha}\|_{L^{\infty}}\leq \frac{C}{\sqrt{\alpha}}$ and as $\varphi$ is a smooth compactly supported function, we deduce that
\begin{equation*}
    |K_\alpha| \underset{\alpha\rightarrow \infty}{\longrightarrow} 0.
\end{equation*}
This ends the proof of the first assertion. For the second assertion, we write
\beqn
\Int_{\R^4}\left({\rm e}^{32\pi^2|f_\alpha(x)|^2}-1\right)\varphi(x)\;dx&=&L_\alpha+M_\alpha+N_\alpha,
\eeqn
where
\begin{eqnarray*}
  L_\alpha &=& \int_{|x|\leq {\rm{e}}^{-\alpha}
  } \left({\rm{e}}^{32\pi^2|f_{\alpha}(x)|^2}-1\right) \varphi(x)\; dx,\\
  M_\alpha &=& \int_{{\rm{e}}^{-\alpha}\leq |x| \leq 1} \left({\rm{e}}^{32\pi^2 |f_{\alpha}(x)|^2}-1\right) \varphi(x)\; dx\quad\mbox{and}\\
 N_\alpha &=&\int_{ |x| \geq 1} \left({\rm{e}}^{32\pi^2 |f_{\alpha}(x)|^2}-1\right) \varphi(x)\; dx.
\end{eqnarray*}
We have
$$  L_\alpha = \int_{|x|\leq {\rm{e}}^{-\alpha}} \left({\rm{e}}^{32\pi^2|f_{\alpha}(x)|^2}-1\right) \big(\varphi(x)-\varphi(0)\big)\; dx+\int_{|x|\leq {\rm{e}}^{-\alpha}
  } \left({\rm{e}}^{32\pi^2|f_{\alpha}(x)|^2}-1\right) \varphi(0)\; dx.$$
Arguing as above, we infer that
$$\left|L_\alpha-\int_{|x|\leq {\rm{e}}^{-\alpha}
  } \left({\rm{e}}^{32\pi^2|f_{\alpha}(x)|^2}-1\right) \varphi(0)\; dx\right|\leq
2\pi^2\|\nabla \varphi\|_{L^{\infty}} \left({\rm{e}}^{32\pi^2\left({\sqrt{\frac{\alpha}{8\pi^2}}}+\frac{1}{\sqrt{32\pi^2 \alpha}}\right)^2}-1\right)\frac{{\rm{e}}^{-5\alpha}}{5}.$$
As the right hand side of the last inequality goes to zero when $\alpha$ tends to infinity, we find that
\begin{equation*}
    \left|L_\alpha-\int_{|x|\leq {\rm{e}}^{-\alpha}
  } \left({\rm{e}}^{32\pi^2|f_{\alpha}(x)|^2}-1\right) \varphi(0)\; dx\right|\underset{\alpha\rightarrow \infty}\longrightarrow 0.
\end{equation*}

Besides,
\begin{eqnarray*}  \int_{|x|\leq {\rm{e}}^{-\alpha}} \left({\rm{e}}^{32\pi^2|f_{\alpha}(x)|^2}-1\right) \varphi(0)\; dx &=&
2\pi^2{\rm e}^{4(\alpha+1)}{\rm e}^\frac{1}{\alpha}\varphi(0)\Int_0^{{\rm e}^{-\alpha}}{\rm e}^{\frac{{\rm e}^{4\alpha}}{\alpha}r^4-2{\rm{e}}^{2\alpha}(2+\frac{1}{\alpha})r^2}r^3\;dr\\
&-&\frac{\pi^2}{2}\varphi(0){\rm e}^{-4\alpha}.
\end{eqnarray*}
Now, performing the change of variable $s=r{\rm{e}}^{\alpha}$, we get
$$  \int_{|x|\leq {\rm{e}}^{-\alpha}} \left({\rm{e}}^{32\pi^2|f_{\alpha}(x)|^2}-1\right) \varphi(0)\; dx =
2\pi^2{\rm e}^{\frac{1}{\alpha}+4}\varphi(0)\Int_0^1s^3{\rm e}^{\frac{s^4}{\alpha}-2(2+\frac{1}{\alpha})s^2}\;ds\\
-\frac{\pi^2}{2}\varphi(0){\rm e}^{-4\alpha},$$
which implies, in view of Lebesgue's theorem, that
\begin{equation*}
    \underset{\alpha\rightarrow \infty}{\lim} L_\alpha= 2\pi^2{\rm e}^4\varphi(0)\int_{0}^{1} s^3\,{\rm{e}}^{-4s^2}\;ds=\frac{\pi^2}{16}({\rm e}^4-5)\varphi(0).
\end{equation*}
Also, writing
\begin{equation*}
    M_\alpha=\int_{{\rm{e}}^{-\alpha}\leq |x|\leq 1}\big(\varphi(x)-\varphi(0)\big)\Big({\rm{e}}^{\frac{4(\log|x|)^2}{\alpha}}-1\Big)\;dx+\int_{{\rm{e}}^{-\alpha}\leq |x|\leq 1}\varphi(0)\Big({\rm{e}}^{\frac{4(\log|x|)^2}{\alpha}}-1\Big)\; dx,
\end{equation*}
 we infer that $M_\alpha$ converges to $\Frac{\pi^2}{2}\varphi(0)$ by using the following lemma the proof of which is similar to that of Lemma 1.9 in \cite{Bahouri}.
\end{proof}
\begin{lem}
\label{add1}
When $\alpha$ goes to infinity,
$$\Int_{{\rm e}^{-\alpha}}^1\;r^4\,{\rm e}^{\frac{4}{\alpha}\log^2 r}\;
dr\longrightarrow \frac{1}{5}\quad \mbox{and}\quad \Int_{{\rm e}^{-\alpha}}^1\;r^3\,{\rm
e}^{\frac{4}{\alpha}\log^2 r}\; dr\longrightarrow \frac{1}{2}.$$
\end{lem}
Finally, in view of the existence of a positive constant $C$ such taht $\|\eta_{\alpha}\|_{L^\infty}\leq \frac{C}{\sqrt{\alpha}}$ and as $\varphi$ is a smooth compactly supported function, we get
\begin{equation*}
    N_\alpha\underset{\alpha\rightarrow \infty}{\longrightarrow} 0,
\end{equation*}
which achieves the proof of the proposition.
\subsection{Statement of the results}
 Before entering into the details, let us introduce some definitions
as in \cite{Bahouri} and \cite{Gerard}.
\begin{defi} We shall designate by a scale any sequence $\alpha:= (\alpha_n)$ of positive real numbers going to
infinity. Two scales $\alpha$ and $\beta$ are said orthogonal if
\begin{equation*}
    \Big{|}\log\Big(\frac{\beta_n}{\alpha_n}\Big)\Big{|}\rightarrow \infty.
\end{equation*}
The set of profiles is
 \begin{equation*}
    \mathcal{P}:=\Big\{\psi\in L^2(\R,{\rm{e}}^{-4s}ds);\quad\psi'\in L^2(\R),\quad\psi_{|]-\infty,0]}=0\Big\}.
\end{equation*}
\end{defi}
\begin{rem}
 The profiles belong to the H\"older space $C^\frac{1}{2}$. Indeed, for any profile $\psi$ and real numbers $s$ and $t$, we have by Cauchy-Schwarz inequality
$$|\psi(s)-\psi(t)|=\left|\int_s^t\psi'(\tau)\;d\tau\right|\leq\|\psi'\|_{L^2(\R)}|s-t|^\frac{1}{2}.$$
\end{rem}
Our main goal is to establish that the characterization of the lack of compactness of critical Sobolev embedding
\begin{equation*}
    H^2_{rad}(\mathbb{R}^4)\hookrightarrow \mathcal{L}(\mathbb{R}^4)
\end{equation*}
can be reduced to the example \eqref{exem}. In fact, we can decompose the function $f_\alpha$ as follows:
$$f_\alpha(x)=\sqrt{\Frac{\alpha}{8\pi^2}}L\Big(-\Frac{\log|x|}{\alpha}\Big)+r_\alpha(x),$$
where
\begin{equation*}
 L(t)=\left\{%
\begin{array}{lll}
1 & \hbox{if\quad $t\geq 1$} \\
t &\hbox{if\quad $0\leq t<1$}\\
0 &\hbox{if\quad $t< 0$}
\end{array}%
\right.
\end{equation*}
and
\begin{equation*}
 r_\alpha(x)=\left\{%
\begin{array}{lll}
\frac{1-|x|^2{\rm e}^{2\alpha}}{\sqrt{32\pi^2\alpha}} & \hbox{if\quad $|x|\leq {\rm e}^{-\alpha}$} \\
0 &\hbox{if\quad ${\rm e}^{-\alpha}<|x|\leq 1$}\\
\eta_\alpha(x) &\hbox{if\quad $|x|> 1$.}
\end{array}%
\right.
\end{equation*}
The sequence $\alpha$ is a scale, the function $L$ is a profile and the function $r_\alpha$ is called the remainder term.\\
We can easily see that $r_\alpha\underset{\alpha\rightarrow\infty}\longrightarrow 0$ in $\mathcal{L}.$ Indeed, for all $\lambda>0$, we have
\begin{eqnarray*}
\Int_{|x|\leq {\rm e}^{-\alpha}}\Big({\rm e}^\frac{|r_\alpha(x)|^2}{\lambda^2}-1\Big)\;dx
&\leq& 2\pi^2\Int_0^{{\rm e}^{-\alpha}}\Big({\rm e}^\frac{1+r^4{\rm e}^{4\alpha}}{16\pi^2\alpha\lambda^2}-1\Big)r^3\;dr\\
&\leq& \Big[8\pi^4\lambda^2{\rm e}^{\frac{1}{16\pi^2\alpha\lambda^2}}\alpha{\rm e}^{-4\alpha}\Big({\rm e}^{\frac{1}{16\pi^2\alpha\lambda^2}}-1\Big)-\frac{\pi^2{\rm e}^{-4\alpha}}{2}\Big]\underset{\alpha\rightarrow\infty}\longrightarrow 0.
\end{eqnarray*}
Moreover, since $\eta$ belongs to $\mathcal{D}(\R^4)$ and satisfies $\|\eta_{\alpha}\|_{L^\infty}\leq \frac{C}{\sqrt{\alpha}}$ for some $C>0$, we get
$$\Int_{|x|> 1}\Big({\rm e}^\frac{|r_\alpha(x)|^2}{\lambda^2}-1\Big)\;dx\underset{\alpha\rightarrow\infty}\longrightarrow 0.$$
Let us observe that $h_\alpha(x):=\sqrt{\Frac{\alpha}{8\pi^2}}L\Big(-\Frac{\log|x|}{\alpha}\Big)$ does not belong to $H^2(\R^4)$. To overcome this difficulty, we shall convolate the profile $L$ with an approximation to the identity $\rho_n$ where $\rho_n(s)=\alpha_n\rho(\alpha_n s)$ with $\rho$ is a positive smooth compactly supported function satisfying
 \begin{equation}\label{rho}
 \mbox{supp}\,\rho\subset[-1,1]\quad\mbox{ and}
\end{equation}
\begin{equation}\label{int rho}
 \Int_{-1}^1\rho(s)\;ds=1.
\end{equation}
 More precisely, we shall prove that the lack of compactness can be described in terms of an asymptotic decomposition as follows:
\begin{thm}\label{th} Let $(u_n)_n$ be a bounded sequence in $H_{rad}^2(\mathbb{R}^4)$ such that
\begin{eqnarray}
  u_n &\underset{n\rightarrow\infty}\rightharpoonup& 0,\label{hyp1}\\
  \underset{n\rightarrow \infty}{\limsup}\,\|u_n\|_{\mathcal{L}} &=&
A_0>0,\quad\mbox{and}\label{hyp2}\\
  \underset{R\rightarrow \infty}{\lim}\underset{n \rightarrow \infty}{\limsup}\int_{|x|>R}|u_n(x)|^2 \;dx&=& 0\label{hyp3}.
\end{eqnarray}
Then, there exists a sequence $(\alpha^{(j)})$ of pairwise orthogonal scales and a sequence of profiles $(\psi^{(j)})$ in $\mathcal{P}$ such that up to a subsequence extraction, we have for all $\ell\geq 1$
\begin{equation}\label{decomposition}
  u_n(x)=\sum_{j=1}^{\ell}\sqrt{\frac{\alpha_n^{(j)}}{8\pi^2}}\big(\psi^{(j)}*\rho_n^{(j)}\big)\left(\frac{-\log|x|}{\alpha_{n}^{(j)}}\right)+
r_n^{(\ell)}(x),
\end{equation}
where $\rho_n^{(j)}(s)=\alpha_n^{(j)}\rho(\alpha_n^{(j)}s)$ and $\Limsup_{n\rightarrow \infty}\big\|r_n^{(\ell)}\big\|_{\mathcal{L}}\overset{\ell\rightarrow \infty}{\longrightarrow}0$.
\end{thm}
\begin{rems}
{\bf a)} As in \cite{Gerard}, the decomposition \eqref{decomposition} is not unique.\\
{\bf b)} The assumption \eqref{hyp3} means that there is no lack of compactness at
infinity. It is in particularly satisfied when the
sequence $(u_n)$ is supported in a fixed compact
of $\R^4$ and also by the sequences 
\begin{equation}\label{gnj}
g_n^{(j)}(x):=\sqrt{\frac{\alpha_n^{(j)}}{8\pi^2}}\big(\psi^{(j)}*\rho_n^{(j)}\big)\left(\frac{-\log|x|}{\alpha_{n}^{(j)}}\right)
\end{equation}
involved in the decomposition \eqref{decomposition}.\\
{\bf c)} As it is mentioned above, the functions $h_n^{(j)}(x):=\sqrt{\frac{\alpha_n^{(j)}}{8\pi^2}}\psi^{(j)}\left(\frac{-\log|x|}{\alpha_{n}^{(j)}}\right)$ do not belong to $H^2(\R^4)$. However, we have 
\begin{equation}\label{gn hn}
\big\|g_n^{(j)}-h_n^{(j)}\big\|_{\mathcal{L}(\R^4)}\underset{n\rightarrow\infty}\longrightarrow 0,
\end{equation}
where the functions $g_n^{(j)}$ are defined by \eqref{gnj}. Indeed, by the change of variable $s=-\frac{\log |x|}{\alpha_n^{(j)}}$ and using the fact that, for any integer number $j$, $\psi^{(j)}*\rho^{(j)}_n$ is supported in $[-\frac{1}{\alpha_n^{(j)}},\infty[$ and $\psi^{(j)}$ is supported in $[0,\infty[$, we infer that for all $\lambda>0$
$$\Int_{\R^4}\Big({\rm e}^{\big|\frac{g_n^{(j)}(x)-h_n^{(j)}(x)}{\lambda}\big|^2}-1\Big)\; dx=2\pi^2\alpha_n^{(j)}\Int_{-\frac{1}{\alpha_n^{(j)}}}^\infty\Big({\rm e}^{\frac{\alpha_n^{(j)}}{8\pi^2\lambda^2}\big|\big(\psi^{(j)}*\rho_n^{(j)}\big)(s)-\psi^{(j)}(s)\big|^2}-1\Big){\rm e}^{-4\alpha_n^{(j)}s}\; ds.$$
Since
$$\big|\big(\psi^{(j)}*\rho_n^{(j)}\big)(s)-\psi^{(j)}(s)\big|\leq\Int_{-1}^1\Big|\psi^{(j)}\Big(s-\frac{t}{\alpha_n^{(j)}}\Big)-\psi^{(j)}(s)\Big|\rho(t)\;dt,$$
we obtain, according to Cauchy-Schwarz inequality,
\begin{eqnarray*}
\big|\big(\psi^{(j)}*\rho_n^{(j)}\big)(s)-\psi^{(j)}(s)\big|^2&\lesssim&\Int_{-1}^1 \Big|\psi^{(j)}\Big(s-\frac{t}{\alpha_n^{(j)}}\Big)-\psi^{(j)}(s)\Big|^2\;dt\\
&\lesssim&\alpha_n^{(j)}\Int_{-\frac{1}{\alpha_n^{(j)}}}^\frac{1}{\alpha_n^{(j)}}\big|\psi^{(j)}(s-\tau)-\psi^{(j)}(s)\big|^2\;d\tau\\
&\lesssim& \alpha_n^{(j)}\Int_{-\frac{1}{\alpha_n^{(j)}}}^\frac{1}{\alpha_n^{(j)}}\Big(\Int_{s-\tau}^s\big|\big(\psi^{(j)}\big)'(u)\big|\;du\Big)^2\;d\tau.
\end{eqnarray*}
Applying again Cauchy-Schwarz inequality, we get
\begin{eqnarray*}
\big|\big(\psi^{(j)}*\rho_n^{(j)}\big)(s)-\psi^{(j)}(s)\big|^2&\lesssim&\alpha_n^{(j)}\Int_{-\frac{1}{\alpha_n^{(j)}}}^\frac{1}{\alpha_n^{(j)}}\Big(\Int_{s-\tau}^s\big|\big(\psi^{(j)}\big)'(u)\big|^2\;du\Big)|\tau|\;d\tau\\
&\lesssim& \frac{1}{\alpha_n^{(j)}}\Sup_{|\tau|\leq\frac{1}{\alpha_n^{(j)}}}\Int_{s-\tau}^s\big|\big(\psi^{(j)}\big)'(u)\big|^2\;du.
\end{eqnarray*}
Then, there exists a positive constant $C$ such that
\begin{eqnarray*}
\Int_{\R^4}\Big({\rm e}^{\big|\frac{g_n^{(j)}(x)-h_n^{(j)}(x)}{\lambda}\big|^2}-1\Big)\; dx
&\lesssim&\alpha_n^{(j)}\Int_{-\frac{1}{\alpha_n^{(j)}}}^\infty\left({\rm e}^{\frac{C}{\lambda^2}\underset{|\tau|\leq\frac{1}{\alpha_n^{(j)}}}{\sup}\int_{s-\tau}^s|(\psi^{(j)})'(u)|^2\;du}-1\right){\rm e}^{-4\alpha_n^{(j)}  s}\;ds\\
&\lesssim& I_n+J_n,
\end{eqnarray*}
where 
\begin{eqnarray*}
I_n&=&\alpha_n^{(j)}\Int_{s_0}^\infty\left({\rm e}^{\frac{C}{\lambda^2}\underset{s\in[s_0,\infty[,|\tau|\leq\frac{1}{\alpha_n^{(j)}}}{\sup}\int_{s-\tau}^s|(\psi^{(j)})'(u)|^2\;du}-1\right){\rm e}^{-4\alpha_n^{(j)}  s}\;ds\quad \mbox{and}\\
J_n&=&\alpha_n^{(j)}\Int_{-\frac{1}{\alpha_n^{(j)}}}^{s_0}\left({\rm e}^{\frac{C}{\lambda^2}\underset{s\in[-\frac{1}{\alpha_n^{(j)}},{s_0}],|\tau|\leq\frac{1}{\alpha_n^{(j)}}}{\sup}\int_{s-\tau}^s|(\psi^{(j)})'(u)|^2\;du}-1\right){\rm e}^{-4\alpha_n^{(j)}  s}\;ds,
\end{eqnarray*}
for some positive real $s_0$.\\
Noticing that
$$I_n\lesssim\left({\rm e}^{\frac{C\left\|\left(\psi^{(j)}\right)'\right\|_{L^2(\R)}^2}{\lambda^2}}-1\right)\frac{{\rm e}^{-4\alpha_n^{(j)}s_0}}{4},$$
we infer that
$$\Lim_{n\rightarrow\infty}I_n=0.$$
Moreover, the fact that
$$C_n:=C\,\Sup_{s\in[-\frac{1}{\alpha_n^{(j)}},s_0],|\tau|\leq\frac{1}{\alpha_n^{(j)}}}\Int_{s-\tau}^s\big|\big(\psi^{(j)}\big)'(u)\big|^2\;du\\
\underset{n\rightarrow\infty}\longrightarrow 0,$$
implies that 
$$\Lim_{n\rightarrow\infty}J_n=\Lim_{n\rightarrow\infty}\Big({\rm e}^\frac{C_n}{\lambda^2}-1\Big)\frac{{\rm e}^4-{\rm e}^{-4\alpha_n^{(j)}s_0}}{4}= 0.$$
This leads to \eqref{gn hn} as desired.\\
{\bf d)} Similarly to the proof of Proposition 1.15 in \cite{Bahouri}, we get by using \eqref{gn hn} 
$$\Lim_{n\rightarrow\infty}\big\|g_n^{(j)}\big\|_{\mathcal{L}(\R^4)}=\Lim_{n\rightarrow\infty}\big\|h_n^{(j)}\big\|_{\mathcal{L}(\R^4)}=\Frac{1}{\sqrt{32\pi^2}}\Max_{s>0}\Frac{\big|\psi^{(j)}(s)\big|}{\sqrt{s}}.$$
{\bf e)} Setting $\tilde{g}_n(x):=\sqrt{\frac{\alpha_n^{(j)}}{8\pi^2}}\big(\psi^{(j)}*\tilde{\rho}_n^{(j)})\Big(\frac{-\log|x|}{\alpha_n^{(j)}}\Big),$
where $\tilde{\rho}_n^{(j)}(s)=\alpha_n^{(j)}\tilde{\rho}\big(\alpha_n^{(j)}s\big)$ with $\tilde{\rho}$ is a positive smooth compactly supported function satisfying \eqref{rho} and \eqref{int rho}, we notice that
\begin{equation}\label{g_n}
\big\|g_n^{(j)}-\tilde{g}_n^{(j)}\big\|_{\mathcal{L}(\R^4)}\underset{n\rightarrow\infty}\longrightarrow 0,
\end{equation}
where the functions $g_n^{(j)}$ are defined by \eqref{gnj}. To prove \eqref{g_n}, we apply the same lines of reasoning of the proof of \eqref{gn hn}.\\ 
{\bf f)} Compared with the decomposition in \cite{Gerard}, it can be seen that there's no core in \eqref{decomposition}. This is justified by the radial setting.
\end{rems}
 Theorem \ref{th} induces to
$$\|u_n\|_{\mathcal{L}}\rightarrow\Sup_{j\geq 1}\Big(\underset{n\rightarrow \infty}{\lim}\big\|g_n^{(j)}\big\|_{\mathcal{L}}\Big).$$
This is due to the following proposition proved in \cite{Bahouri}.
\begin{prop}\label{conv.prop}
\label{sumOrlicz} Let $({\alpha}^{(j)})_{1\leq
j\leq\ell}$ be a family of pairwise orthogonal scales and
$(\psi^{(j)})_{1\leq j\leq\ell}$ be a family of profiles, and set
\begin{eqnarray*}
    g_n(x)&=&\sum_{j=1}^{\ell}\,\sqrt{\frac{\alpha_n^{(j)}}{8\pi^2}}\;\big(\psi^{(j)}*\rho_n^{(j)}\big)\left(\frac{-\log|x|}{\alpha_n^{(j)}}\right):=\sum_{j=1}^{\ell}\,g_n^{(j)}(x)\;.
  \end{eqnarray*}
 Then
\begin{equation*}
 \|g_n\|_{\mathcal{L}}\to \Sup_{1\leq
j\leq\ell}\,\left(\lim_{n\to\infty}\,\big\|g_n^{(j)}\big\|_{\mathcal{L}}\right).
\end{equation*}
\end{prop}
\subsection{Structure of the paper} The paper is organized as follows: Section 2 is devoted to the proof of Theorem \ref{th} by describing the algorithm construction of the decomposition of a bounded sequence $(u_n)$ in $H^2_{rad}(\R^4)$, up a subsequence extraction, in terms of orthogonal profiles. In the last section, we deal with several complements for the sake of completeness.\\
We mention that $C$ will be used to denote a constant which may vary from line to line. We also use $A \lesssim B$ to denote an estimate of the form $A \leq CB$ for some absolute constant $C$ and
$A \approx B$ if $A \lesssim B$ and $B \lesssim A$. For simplicity, we shall also still denote by $(u_n)$ any subsequence
of $(u_n)$.
\section{Proof of the main theorem}
\subsection{Scheme of the proof}
The first step of the proof is based on the extraction of the first scale and the first profile. As in \cite{Bahouri}, the heart of the matter is reduced to the proof of the following lemma:
\begin{lem}
Let $(u_n)$ be a sequence in $H^2_{rad}(\R^4)$ satisfying the assumptions of Theorem \ref{th}. Then there exists a scale $(\alpha_n)$ and a profile $\psi$ such that
\begin{equation}\label{key}
\|\psi'\|_{L^2(\R)}\geq C A_0,
\end{equation}
where $C$ is a universal constant.
\end{lem}
Then, the problem will be reduced to the study of the remainder term. If the limit of
its Orlicz norm is null we stop the process. If not, we prove that this remainder
term satisfies the same properties as the sequence start which allows us to apply the
lines of reasoning of the first step and extract a second scale and a second profile which verify the above key property \eqref{key}. By contradiction arguments, we get the
property of orthogonality between the two first scales. Finally, we prove that this process converges.
\subsection{Preliminaries}
To describe the lack of compactness of the Sobolev space $H^2_{rad}(\R^4)$ into the Orlicz space $\mathcal{L}(\mathbb{R}^4)$, we will make firstly the change of variable $s:=-\log r$ with $r=|x|$ and associate to any radial function $u$ on $\R^4$ a one
space variable function $v$ defined by $v(s) = u(\rm{e}^{-s})$. It follows that:
\begin{eqnarray}
\|u\|_{L^2(\R^4)}^2&=&2\pi^2 \int_{\mathbb{R}} {\rm{e}}^{-4s} |v(s)|^2\; ds\label{C1},\\
\Big{\|}\frac{{\partial} u}{{\partial} r}\Big{\|}_{L^2(\R^4)}^2&=&2\pi^2 \int_{\mathbb{R}} {\rm{e}}^{-2s}|v'(s)|^2 \;ds,\label{C2}\\
\Big{\|}\frac{1}{r} {\partial}_r u\Big{\|}_{L^2(\R^4)}^2&=&2\pi^2 \int_{\mathbb{R}}|v'(s)|^2\;ds \label{C4}\quad and \\
\|\Delta u\|_{L^2(\R^4)}^2&=&2\pi^2 \int_{\mathbb{R}} |-2 v'(s)+v''(s)|^2 \;ds.\label{C5}
\end{eqnarray}
The quantity \eqref{C4} will play a fondamental role in our main result. Moreover, for a scale $(\alpha_n)$ and a profile $\psi$ we define
\begin{equation*}  g_{n}(x):=\sqrt{\frac{\alpha_n}{8\pi^2}}\,(\psi*\rho_n)\left(\frac{-\log|x|}{\alpha_n}\right),
\end{equation*}
where $\rho_n(s)=\alpha_n\rho(\alpha_n s)$ with $\rho$ is a positive smooth compactly supported function satisfying \eqref{rho} and \eqref{int rho}.
Straightforward computations show that
        \begin{eqnarray}
             \|g_n\|_{L^2(\R^4)} &\lesssim&\alpha_n\left(\Int^{\infty}_{0}|\psi(s)|^2{\rm e}^{-4\alpha_n s}\;ds\right)^\frac{1}{2} \label{p1}, \\
             \Big{\|}\frac{\partial g_n}{\partial r}\Big{\|}_{L^2(\R^4)} &\lesssim&\left(\Int_{\R}|\psi'(s)|^2{\rm e}^{-2\alpha_n s}\;ds\right)^\frac{1}{2} ,\label{p2} \\
   \Big\|\frac{1}{r}\partial_rg_n\Big\|_{L^2(\R^4)}&\lesssim& \|\psi'\|_{L^2(\R)}\label{p4} \quad and\\
             \|\Delta g_n\|_{L^2(\R^4)} &\lesssim& \|\psi'\|_{L^2(\R)}.\label{p5}
\end{eqnarray}
Indeed, we have
\begin{eqnarray*}
\|g_n\|_{L^2(\R^4)}&=&\frac{\alpha_n}{2}\left(\Int_\R|(\psi*\rho_n)(s)|^2{\rm e}^{-4\alpha_n s}\;ds\right)^\frac{1}{2}\\
&=&\big\|\tilde{\psi}_n*\tilde{\rho}_n\big\|_{L^2(\R)},
\end{eqnarray*}
where $\tilde{\psi}_n(\tau)=\frac{\alpha_n}{2}\psi(\tau){\rm e}^{-2\alpha_n\tau}$ and $\tilde{\rho}_n(\tau)=\rho_n(\tau){\rm e}^{-2\alpha_n\tau}$. According to Young's inequality, we get
 $$\|g_n\|_{L^2(\R^4)}\leq \big\|\tilde{\psi}_n\big\|_{L^2(\R)}\|\tilde{\rho}_n\|_{L^1(\R)}.$$
Since $\big\|\tilde{\psi}_n\big\|_{L^2(\R)}=\frac{\alpha_n}{2}\left(\Int_{0}^{\infty}|\psi(\tau)|^2{\rm e}^{-4\alpha_n \tau}\;d\tau\right)^\frac{1}{2}$ and $\|\tilde{\rho}_n\|_{L^1(\R)}=\Int_{-1}^1\rho(\tau){\rm e}^{-2\tau} \;d\tau$, we obtain \eqref{p1}.\\
 Similarly, writing
\begin{eqnarray*}
 \Big{\|}\frac{\partial g_n}{\partial r}\Big{\|}_{L^2(\R^4)}&=&
\frac{1}{2}\left(\int_\R|(\psi'*\rho_n)(s)|^2{\rm e}^{-2\alpha_ns}\;ds\right)^\frac{1}{2}\\
&=&\Big\|\tilde{\tilde{\psi}}_n*\tilde{\tilde{\rho}}_n\Big\|_{L^2(\R)},
\end{eqnarray*}
where $\tilde{\tilde{\psi}}_n(\tau)=\frac{1}{2}\psi'(\tau){\rm e}^{-\alpha_n\tau}$ and $\tilde{\tilde{\rho}}_n(\tau)=\rho_n(\tau){\rm e}^{-\alpha_n\tau}$ and using Young's inequality, we infer that
\begin{eqnarray*}
\Big{\|}\frac{\partial g_n}{\partial r}\Big{\|}_{L^2(\R^4)}&\leq&\Big\|\tilde{\tilde{\psi}}\Big\|_{L^2(\R)}\Big\|\tilde{\tilde{\rho}}\Big\|_{L^1(\R)}\\
&\leq&\frac{1}{2}\left(\Int_{\R}|\psi'(\tau)|^2{\rm e}^{-2\alpha_n\tau}\;d\tau\right)^\frac{1}{2}\int_{-1}^1\rho(\tau){\rm e}^{-\tau}\;d\tau,
\end{eqnarray*}
which leads to \eqref{p2}.\\
Also, we have
$$\Big\|\frac{1}{r}\partial_rg_n\Big\|_{L^2(\R^4)}=\frac{1}{2}\|\psi'*\rho_n\|_{L^2(\R)}\leq \frac{1}{2}\|\psi'\|_{L^2(\R)}.$$
Finally,
\begin{eqnarray*}
\|\Delta g_n\|_{L^2(\R^4)}
&=&\frac{1}{2}\left(\Int_\R\Big|-2(\psi'*\rho_n)(s)+\frac{1}{\alpha_n}(\psi'*\rho_n')(s)\Big|^2\;ds\right)^\frac{1}{2}\\
&\leq& \|\psi'*\rho_n\|_{L^2(\R)}+\frac{1}{2\alpha_n} \|\psi'*\rho_n'\|_{L^2(\R)}\\
&\leq&\|\psi'\|_{L^2(\R)}+\frac{1}{2\alpha_n}\|\psi'\|_{L^2(\R)}\|\rho'_n\|_{L^1(\R)}.
\end{eqnarray*}
The fact that $\|\rho'_n\|_{L^1(\R)}=\alpha_n\Int_{-1}^1\rho'(\tau)\;d\tau$ ensures \eqref{p5}.
\subsection{Extraction of the first scale and the first profile}
Let us consider a bounded sequence $(u_n)$ in $H^2_
{rad}(\R^4)$ satisfying the assumptions \eqref{hyp1}, \eqref{hyp2} and \eqref{hyp3} and let us set $$v_n(s):=u_n(e^{-s}).$$ We have the following lemma.
\begin{lem}\label{lem1}
Under the above assumptions, the sequence $(u_n)$ converges strongly to 0 in $L^2(\R^4)$. Moreover, for any real number $M$, we have
\begin{equation}\label{conv vn} \Lim_{n\rightarrow\infty}\|v_n\|_{L^\infty(]-\infty,M[)}=0.
\end{equation}
\end{lem}
\begin{proof}
 For any $R>0$, we have
$$\|u_n\|_{L^2(\R^4)}=\|u_n\|_{L^2(|x|<R)}+\|u_n\|_{L^2(|x|>R)}.$$
According to Rellich's theorem, the Sobolev space $H^2(|x| < R)$ is compactly embedded in $L^2(|x| < R)$. Thanks to \eqref{hyp1}, we get $$\Lim_{n\rightarrow\infty}\|u_n\|_{L^2(|x|<R)}=0.$$
Now, taking advantage of the compactness at infinity of the sequence $(u_n)$ given by \eqref{hyp3}, we deduce that
\begin{equation}\label{lim un}
\Lim_{n\rightarrow\infty}\|u_n\|_{L^2(\R^4)}=0.
\end{equation}
Besides, according to Proposition \ref{est.rad}, we infer that
\begin{equation}\label{bound vn}
|v_n(s)|\lesssim {\rm e}^{\frac{3}{2}s}\|u_n\|_{L^2(\R^4)}^\frac{1}{2}\|\nabla u_n\|_{L^2(\R^4)}^\frac{1}{2}.
\end{equation}
For $s<M$, \eqref{conv vn} derives immediately from \eqref{bound vn} and the
strong convergence of $(u_n)$ to zero in $L^2(\R^4)$.
\end{proof}

Now, we shall determine the first scale and the first profile.
\begin{prop}\label{pro1}
For all $0<\delta<A_0$, we have
$$\Sup_{s\geq 0}\left(\Big|\frac{v_n(s)}{A_0-\delta}\Big|^2-3s\right)\underset{n\rightarrow\infty}\longrightarrow \infty.$$
\end{prop}
\begin{proof}
To go to the proof of Proposition \ref{pro1}, we shall proceed by contradiction by assuming that there exists a positive real $\delta$ such that, up to a subsequence extraction,
\begin{equation}\label{cont vn}
\Sup_{s\geq 0,n\in\N}\left(\Big|\frac{v_n(s)}{A_0-\delta}\Big|^2-3s\right)\leq C,
\end{equation}
 where $C$ is a positive constant. Thanks to \eqref{conv vn} and \eqref{cont vn}, we get by virtue of Lebesgue's theorem
$$\Lim_{n\rightarrow\infty}\Int_{|x|<1}\Big({\rm e}^{\big|\frac{u_n(x)}{A_0-\delta}\big|^2}-1\Big)\;dx=\Lim_{n\rightarrow\infty}2\pi^2\Int_0^\infty\Big({\rm e}^{\big|\frac{v_n(s)}{A_0-\delta}\big|^2}-1\Big){\rm e}^{-4s}\;ds=0.$$
On the other hand, using Proposition \ref{est.rad}, the boundedness of $(u_n)$ in $H^2(\R^4)$ ensures the existence of a positive constant $C$ such that
$$|u_n(x)|\leq C,\quad \forall\;n\in\N\;\mbox{ and}\;|x|\geq 1.$$
By virtue of the fact that for any positive $M$ there exists a
finite constant $C_M$ such that
$$\Sup_{|t|\leq M}\Big(\Frac{{\rm e}^{t^2}-1}{t^2}\Big)<C_M,$$
we obtain that
$$\Int_{|x|\geq 1}\Big({\rm e}^{\big|\frac{u_n(x)}{A_0-\delta}\big|^2}-1\Big)\;dx\leq C\|u_n\|_{L^2(\R^4)}^2.$$
The strong convergence of $(u_n)$ to 0 in $L^2(\R^4)$ leads to
$$\Int_{\R^4}\Big({\rm e}^{\big|\frac{u_n(x)}{A_0-\delta}\big|^2}-1\Big)\;dx\underset{n\rightarrow \infty}\longrightarrow 0.$$
Thus, $$\Lim_{n\rightarrow \infty}\|u_n\|_\mathcal{L}\leq A_0-\delta,$$
which is in contradiction with Hypothesis \eqref{hyp2}.
\end{proof}
\begin{cor}\label{cor1}
There exists a scale $\big(\alpha_n^{(1)}\big)$ such that
$$4\left|\frac {v_n\big(\alpha_n^{(1)}\big)}{A_0} \right|^2-3\,\alpha_n^{(1)}\underset{n\rightarrow\infty}\longrightarrow \infty.$$
\end{cor}
\begin{proof}
Let us set
$$W_n(s):=4\left|\frac {v_n(s)}{A_0} \right|^2-3s\quad and\quad a_n:=\Sup_{s\geq 0}W_n(s).$$
Then, there exists a positive sequence $\big(\alpha_n^{(1)}\big)$ such that
$$W_n\big(\alpha_n^{(1)}\big)\geq a_n-\frac{1}{n}.$$
According to Proposition \ref{pro1}, $a_n$ tends to infinity and then
$$W_n\big(\alpha_n^{(1)}\big)\underset{n\rightarrow\infty}\longrightarrow \infty.$$
It remains to show that $\alpha_n^{(1)}\underset{n\rightarrow\infty}\longrightarrow \infty.$ If not, up to a subsequence extraction, the sequence $\big(\alpha_n^{(1)}\big) $ is bounded in $\R$ and so is $\Big(W_n\big(\alpha_n^{(1)} \big)\Big)$ thanks to \eqref{conv vn}. This yields a contradiction.
\end{proof}
\begin{cor}\label{cor2}
Under the above assumptions, we have for $n$ big enough,
$$\frac{\sqrt{3}}{2}A_0\sqrt{\alpha_n^{(1)}}\leq\big|v_n\big(\alpha_n^{(1)}\big)\big|\leq C\sqrt{\alpha_n^{(1)}}+o(1),$$
where $C=\Frac{1}{\sqrt{8\pi^2}}\Limsup_{n\rightarrow\infty}\|\Delta u_n\|_{L^2(\R^4)}.$
\end{cor}
\begin{proof} The left hand side inequality follows directly from Corollary \ref{cor1}.
On the other hand, for any $s\geq 0$ and according to Cauchy-Schwarz inequality, we obtain that
$$|v_n(s)|=\Big|v_n(0)+\Int_0^s v'_n(\tau)\;d\tau\Big|\leq|v_n(0)|+\sqrt{s}\|v'_n\|_{L^2(\R)}.$$
By virtue of \eqref{C4} and Lemma \ref{lem4}, we get
 $$\|v'_n\|_{L^2(\R)}=\Big(\Int_0^\infty \Big|\frac{1}{r}u'_n(r)\Big|^2r^3\;dr\Big)^\frac{1}{2}\leq \frac{1}{\sqrt{8\pi^2}}\|\Delta u_n\|_{L^2(\R^4)}.$$
Using the boundedness of the sequence $(\Delta u_n)$ in $L^2(\R^4)$ and the convergence of $\big(v_n(0)\big)$ to zero, we infer that
$$|v_n(s)|\leq o(1)+C\sqrt{s},$$
where $C=\Frac{1}{\sqrt{8\pi^2}}\Limsup_{n\rightarrow\infty}\|\Delta u_n\|_{L^2(\R^4)},$ which ensures the right hand side inequality.
\end{proof}
Now we are able to extract the first profile. To do so, let us set
 $$\psi_n(y):=\sqrt{\Frac{8\pi^2}{\alpha_n^{(1)}}}v_n\big(\alpha_n^{(1)}y\big).$$
The following lemma summarizes the principle properties of $\psi_n$.
\begin{lem}\label{lem}
Under the same assumptions, we have
\begin{equation}\label{psi_n(1)}
\sqrt{6\pi^2}A_0\leq|\psi_n(1)|\leq C+o(1),
\end{equation}
where $C=\Limsup_{n\rightarrow\infty}\|\Delta u_n\|_{L^2(\R^4)}.$
Moreover, there exists a profile $\psi^{(1)}$ such that, up to a subsequence extraction, $$\psi'_n\underset{n\rightarrow\infty}\rightharpoonup(\psi^{(1)})'\;in\;L^2(\R)\quad and\quad \big\|(\psi^{(1)})'\big\|_{L^2(\R)}\geq \sqrt{6\pi^2}A_0.$$
\end{lem}
 \begin{proof}
According to Corollary \ref{cor2}, we get \eqref{psi_n(1)}. Besides, thanks to \eqref{C4} and Lemma \ref{lem4} we obtain that
$$\|\psi'_n\|_{L^2(\R)}= \sqrt{8\pi^2}\Big(\Int_0^\infty \Big|\frac{1}{r}u'_n(r)\Big|^2r^3\;dr\Big)^\frac{1}{2}\leq\|\Delta u_n\|_{L^2(\R^4)}.$$
Then, $(\psi'_n)$ is bounded in $L^2(\R)$. Consequently, up to a subsequence extraction, $(\psi'_n)$ converges weakly in $L^2(\R)$ to some function $g\in L^2(\R)$. Let us introduce the function
$$\psi^{(1)}(s):=\Int_0^s g(\tau)\;d\tau.$$
It's obvious that, up asubsequence extraction, $\psi'_n\rightharpoonup(\psi^{(1)})'$ in $L^2(\R)$. It remains to prove that $\psi^{(1)}$ is a profile.\\ Firstly, since $$\big{|}\psi^{(1)}(s)\big{|}=\Big|\Int_0^sg(\tau)\; d\tau\Big|\leq \sqrt{s} \|g\|_{L^2(\R)},$$ we get $\psi^{(1)}\in L^2(\R_+,{\rm e}^{-4s}ds).$\\
Secondly, $\psi^{(1)}(s)=0$ for all $s\leq 0$. Indeed, using the fact that $$\|u_n\|_{L^2(\R^4)}^2=\Frac{\big(\alpha_n^{(1)}\big)^2}{4}\Int_\R |\psi_n(s)|^2 {\rm e}^{-4\alpha_n^{(1)}s}\;ds,$$
we obtain that
$$\Int_{-\infty}^0|\psi_n(s)|^2\;ds\leq\Int_{-\infty}^0|\psi_n(s)|^2 {\rm e}^{-4\alpha_n^{(1)}s}\;ds\leq \Frac{4}{\big(\alpha_n^{(1)}\big)^2}\|u_n\|_{L^2(\R^4)}^2.$$
By virtue of the boundedness of $(u_n)$ in $L^2(\R^4)$, we deduce that $\psi_n$ converges strongly to zero in $L^2(]-\infty,0[).$ Consequently, for almost all $s\leq 0$, up to a subsequence extraction, $\big(\psi_n(s)\big)$ goes to zero. In other respects, as $(\psi'_n)$ converges weakly to $g$ in $L^2(\R)$ and $\psi_n$ belongs to $H^1_{loc}(\R)$, we infer that
\begin{equation*}
\psi_n(s)-\psi_n(0)=\Int_0^s \psi'_n(\tau)\;d\tau\underset{n\rightarrow \infty}\longrightarrow \Int_0^s g(\tau)\; d\tau=\psi^{(1)}(s).
\end{equation*}
This gives rise to the fact that
\begin{equation}\label{a}
\psi_n(s)\underset{n\rightarrow \infty}\longrightarrow \psi^{(1)}(s),\quad\forall\;s\in \R,
\end{equation}
and ensures that ${\psi^{(1)}}_{|]-\infty,0]}=0$.\\
 Finally, knowing that $\big|\psi^{(1)}(1)\big|\geq\sqrt{6\pi^2}A_0$ and
$$\big\|(\psi^{(1)})'\big\|_{L^2(\R)}\geq \Int_0^1\big|(\psi^{(1)})'(\tau)\big|\;d\tau=\big|\psi^{(1)}(1)\big|,$$
we deduce that $\big\|(\psi^{(1)})'\big\|_{L^2(\R)}\geq \sqrt{6\pi^2}A_0$.
\end{proof}

Let us now consider the first remainder term:
\begin{equation}\label{def}
r_n^{(1)}(x)=u_n(x)-g_n^{(1)}(x),
\end{equation}
where $$g_n^{(1)}(x)=\sqrt{\frac{\alpha_n^{(1)}}{8\pi^2}}\big(\psi^{(1)}\ast\rho_{n}^{(1)}\big)\left(\frac{-\log|x|}
{\alpha_n^{(1)}}\right)$$
with $\rho_{n}^{(1)}(s)=\big(\alpha_n^{(1)}\big)\rho \big(\alpha_n^{(1)}s\big)$. Recalling  that $u_n(x)=\sqrt{\frac{\alpha_n^{(1)}}{8\pi^2}}\psi_n\left(\frac{-\log|x|}{\alpha_n^{(1)}}\right)$ and
taking advantage of the fact that $(\psi_n')$ converges weakly in $L^2(\R)$ to $(\psi^{(1)})'$, we get the following result.
\begin{prop}
Let $(u_n)_n$ be a sequence in $H_{rad}^2(\mathbb{R}^4)$ satisfying the assumptions of Theorem \ref{th}. Then, there exist a scale $\big(\alpha_n^{(1)}\big)$ and a profile $\psi^{(1)}$ such
that
\begin{equation}\label{psi}
 \big\|(\psi^{(1)})'\big\|_{L^2(\R)}\geq \sqrt{6\pi^2} A_0.
\end{equation}
In addition, we have
\begin{equation}\label{r1}
\Lim_{n\rightarrow \infty}\Big\|\frac{1}{r} {\partial}_r r_n^{(1)}\Big\|_{L^2(\R^4)}^2=\underset{n\rightarrow \infty}{\lim}\Big\|\frac{1}{r} {\partial}_r u_n\Big\|_{L^2(\R^4)}^2-\frac{1}{4}\big\|(\psi^{(1)})'\big\|_{L^2(\R)}^2,
\end{equation}
where $r_n^{(1)}$ is given by \eqref{def}.
\end{prop}
\begin{proof}
The inequality \eqref{psi} is contained in Lemma \ref{lem}. Besides, noticing that
$$\Big\|\frac{1}{r} {\partial}_r r_n^{(1)}\Big\|_{L^2(\mathbb{R}^4)}=\frac{1}{2}\big\|\psi'_n-\big((\psi^{(1)})'\ast \rho^{(1)}_n\big)\big\|_{L^2(\mathbb{R})},$$
we get
\begin{eqnarray*}
  \underset{n\rightarrow \infty}{\lim}\Big\|\frac{1}{r} {\partial}_r r_n^{(1)}\Big\|_{L^2(\mathbb{R}^4)}^2
 &=& \frac{1}{4}\underset{n\rightarrow \infty}{\lim}\|\psi'_n\|_{L^2(\R)}^2+\frac{1}{4}\underset{n\rightarrow \infty}{\lim}\big\|({\psi}^{(1)})'\ast\rho^{(1)}_n\big\|_{L^2(\R)}^2\\
&-&\frac{1}{2}\underset{n\rightarrow \infty}{\lim}\int_{\mathbb{R}} \psi'_n(s)\big((\psi^{(1)})'\ast\rho^{(1)}_n\big)(s)\;ds\\
&=& \underset{n\rightarrow \infty}{\lim}\Big\|\frac{1}{r} {\partial}_r u_n\Big\|_{L^2(\mathbb{R}^4)}^2+\frac{1}{4}\big\|({\psi}^{(1)})'\big\|_{L^2(\R)}^2\\
&-&\frac{1}{2}\underset{n\rightarrow \infty}{\lim}\int_{\mathbb{R}} \psi'_n(s)\big((\psi^{(1)})'\ast\rho^{(1)}_n\big)(s)\;ds.
\end{eqnarray*}
We write
\begin{eqnarray*}
  \Int_\R \psi'_n(s)\big((\psi^{(1)})'\ast \rho^{(1)}_n\big)(s)\; ds &=& \Int_\R \psi'_n(s)\Big[\big((\psi^{(1)})'\ast\rho_n^{(1)}\big)(s)-(\psi^{(1)})'(s)\Big]\; ds\\
&+& \Int_\R \psi'_n (s)(\psi^{(1)})'(s)\;ds.
\end{eqnarray*}
Since $(\psi_n')$ converges weakly in $L^2(\R)$ to $(\psi^{(1)})'$, we obtain that
\begin{equation}\label{conv1}
\Int_\R \psi'_n (s)(\psi^{(1)})'(s)\;ds\underset{n\rightarrow\infty}\longrightarrow \big\|(\psi^{(1)})'\big\|_{L^2(\R)}^2.
\end{equation}
Besides, according to Cauchy-Schwarz inequality, we infer that
\begin{eqnarray*}
 \Big| \Int_\R  \psi'_n(s)\Big[\big((\psi^{(1)})'\ast \rho_n^{(1)}\big)(s)-(\psi^{(1)})'(s)\Big]\; ds\Big|
  &\leq& \|\psi'_n\|_{L^2(\R)} \big\|\big((\psi^{(1)})'\ast \rho_n^{(1)}\big)-(\psi^{(1)})'\big\|_{L^2(\R)}\\
  &\leq& 4\Big\|\frac{1}{r} \partial_r u_n\Big\|_{L^2(\R^4)} \big\|\big((\psi^{(1)})'\ast \rho^{(1)}_n\big)-(\psi^{(1)})'\big\|_{L^2(\R)}.
\end{eqnarray*}
The boundedness of $(\frac{1}{r} \partial_r u_n)$ in $L^2(\R^4)$ and the strong convergence of $\big((\psi^{(1)})'\ast \rho^{(1)}_n\big)$ to $(\psi^{(1)})'$ in $L^2(\R)$ imply that
\begin{equation}\label{conv2}
 \Int_\R  \psi'_n(s)\Big[\big((\psi^{(1)})'\ast \rho_n^{(1)}\big)(s)-(\psi^{(1)})'(s)\Big]\; ds\underset{n\rightarrow\infty}\longrightarrow 0.
\end{equation}
 Taking advantage of \eqref{conv1} and \eqref{conv2}, we deduce \eqref{r1}.
\end{proof}
\subsection{Conclusion}
Our concern now is to iterate the previous process and to prove that the algorithmic construction
converges. Thanks to the fact that $\big(\psi^{(1)}*\rho^{(1)}_n\big)$ is supported in $ [-\frac{1}{\alpha_n},\infty[$, we get for $R> {\rm e}$,
\begin{eqnarray*}
 \big\|r_n^{(1)}\big\|_{L^2(|x|> R)}^2 &=& \frac{1}{4}\big(\alpha_n^{(1)}\big)^2\int_{-\infty}^{-\frac{\log R}{\alpha_n^{(1)}}}|\psi_n(t)-\big(\psi^{(1)}\ast\rho^{(1)}_n)(t)\big|^2 {\rm e}^{-4\alpha_n^{(1)}t}\;dt \\
   &=& \frac{1}{4}\big(\alpha_n^{(1)}\big)^2\int_{-\infty}^{-\frac{\log R}{\alpha_n^{(1)}}}|\psi_n(t)|^2 {\rm e}^{-4\alpha_n^{(1)}t}\;dt \\
  &=& \|u_n\|_{L^2(|x|> R)}^2.
\end{eqnarray*}
This implies that $\big(r_n^{(1)}\big)$ satisfies the hypothesis of compactness \eqref{hyp3}. According to \eqref{r1} and the inequalities \eqref{p1}, \eqref{p2} and \eqref{p4}, we deduce that $\big(r_n^{(1)}\big)$ satisfies also \eqref{hyp1}.\\
Let us now define $A_1=\Limsup_{n\rightarrow \infty}\big\|r_n^{(1)}\big\|_{\mathcal{L}}$. If $A_1=0$, we stop the process. If not, since the sequence $\big(r_n^{(1)}\big)$ satisfies the assumptions of Theorem \ref{th}, there exists a scale $\big(\alpha_n^{(2)}\big)$ satisfying the statement of Corollary \ref{cor1} with $A_1$ instead of $A_0$. In particular, there exists a constant $C$ such that
\begin{equation}\label{extr}
    \frac{\sqrt {3}}{2}A_1\sqrt{\alpha_n^{(2)}}\leq \big|\tilde{r}_n^{(1)}\big(\alpha_n^{(2)}\big)\big|\leq C \sqrt{\alpha_n^{(2)}}+o(1),
\end{equation}
where $\tilde{r}_n^{(1)}(s)=r_n^{(1)}({\rm e}^{-s})$. In addition, the scales $\big(\alpha_n^{(1)}\big)$ and $\big(\alpha_n^{(2)}\big)$ are orthogonal. Otherwise, there exists a constant $C$ such that
$$\frac{1}{C}\leq \left|\frac{\alpha_n^{(2)}}{\alpha_n^{(1)}}\right|\leq C.$$
Using \eqref{def}, we get
$$\tilde{r}_n^{(1)}\big(\alpha_n^{(2)}\big)=\sqrt{\frac{\alpha_n^{(1)}}{8\pi^2}}\left(\psi_n\left(\frac{\alpha_n^{(2)}}{\alpha_n^{(1)}}\right)
-\big(\psi^{(1)}\ast\rho^{(1)}_n\big)\left(\frac{\alpha_n^{(2)}}{\alpha_n^{(1)}}\right)\right).$$
For any real number $s$, we have
 $$ \big|\psi_n(s)-\big(\psi^{(1)}\ast\rho^{(1)}_n\big)(s)\big| \leq \big|\psi_n(s)-\psi^{(1)}(s)\big|+\big|\big(\psi^{(1)}\ast\rho^{(1)}_n\big)(s)-\psi^{(1)}(s)\big|. $$
As $\psi^{(1)}$ belongs to the H\"{o}lder space $C^{\frac{1}{2}}$, we obtain that
\begin{eqnarray*}
\big|\big(\psi^{(1)}\ast\rho^{(1)}_n\big)(s)-\psi^{(1)}(s)\big|&=&\Big|\Int_{-\frac{1}{\alpha_n}}^{\frac{1}{\alpha_n}}\rho_n^{(1)}(t)\Big(\psi^{(1)}(s-t)-\psi^{(1)}(s)\Big)\;dt\Big|\\
&\lesssim&\Int_{-\frac{1}{\alpha_n}}^{\frac{1}{\alpha_n}}\rho_n^{(1)}(t)\sqrt{|t|}\;dt\\
&\lesssim& \frac{1}{\sqrt{\alpha_n}}\underset{n\rightarrow\infty}\longrightarrow 0.
\end{eqnarray*}
Thanks to \eqref{a}, we infer that
 $$\big|\psi_n(s)-\big(\psi^{(1)}\ast\rho^{(1)}_n\big)(s)\big|  \underset{n\rightarrow \infty}{\longrightarrow}0.$$
This gives rise to
$$\underset{n\rightarrow\infty}{\lim}\sqrt{\frac{8\pi^2}{\alpha_n^{(1)}}}~\tilde{r}_n^{(1)}\big(\alpha_n^{(2)}\big)=\underset{n\rightarrow \infty}{\lim}\left(\psi_n\left(\frac{\alpha_n^{(2)}}{\alpha_n^{(1)}}\right)-\big(\psi^{(1)}\ast\rho^{(1)}_n\big)\left(\frac{\alpha_n^{(2)}}
{\alpha_n^{(1)}}\right)\right)=0,$$
which is in contradiction with the left hand side inequality of \eqref{extr}.\\
Moreover, there exists a profile $\psi^{(2)}$ such that
$$r_n^{(1)}(x)=\sqrt{\frac{\alpha_n^{(2)}}{8\pi^2}} \big(\psi^{(2)}\ast\rho^{(2)}_n\big)\left(\frac{-\log|x|}{\alpha_n^{(2)}}\right) + r_n^{(2)}(x),$$
where $\rho^{(2)}_n(s)=\alpha_n^{(2)}\rho\big(\alpha_n^{(2)}s\big)$.
Proceeding as the first step, we obtain that $$\big\|(\psi^{(2)})'\big\|_{L^2(\R)}\geq \sqrt{6\pi^2} A_1  \mbox{\, and \,}
\underset{n\rightarrow \infty}{\lim}\Big\|\frac{1}{r} {\partial}_r r_n^{(2)}\Big\|_{L^2(\R^4)}^2=\underset{n\rightarrow \infty}
{\lim}\Big\|\frac{1}{r}{\partial}_r r_n^{(1)}\Big\|_{L^2(\R^4)}^2-\frac{1}{4} \big\|(\psi^{(2)})'\big\|_{L^2(\R)}^2.$$
Consequently,
$$\underset{n\rightarrow \infty}{\lim}\Big\|\frac{1}{r}{\partial}_r r_n^{(2)}\Big\|_{L^2(\R^4)}^2 \leq C-\frac{3\pi^2}{2} A_0^2-\frac{3\pi^2}{2} A_1^2,$$
where $C=\Limsup_{n\rightarrow \infty}\Big \|\Frac{1}{r}{\partial}_r u_n\Big\|_{L^2(\R^4)}^2.$
At iteration $\ell$, we get
\begin{equation*}
    u_n(x)=\sum_{j=1}^{\ell}\sqrt{\frac{\alpha_n^{(j)}}{8\pi^2}}\big(\psi^{(j)}\ast\rho^{(j)}_n\big)\left(\frac{-\log|x|}{\alpha_{n}^{(j)}}\right)+
r_n^{(\ell)}(x),
\end{equation*}
with
$$\Limsup_{\alpha\rightarrow \infty}\Big\|\frac{1}{r}{\partial}_r r_n^{(\ell)}\Big\|_{L^2}^2 \lesssim 1- A_0^2- A_1^2-...-A_{\ell-1}^2.$$
Therefore $A_{\ell}\rightarrow 0$ as $\ell\rightarrow \infty$ and the proof of the main theorem is achieved.
\section{Appendix}
The first part of this appendix presents the proof of the following proposition concerning the convergence in $H^2(\R^4)$ of the sequence $(f_{\alpha})$ defined by \eqref{exem}.
\begin{prop} We have
   $$\|f_{\alpha}\|_{L^2(\R^4)}^2=O\Big(\frac{1}{\alpha}\Big),\quad\|\nabla f_{\alpha}\|_{L^2(\R^4)}^2=O\Big(\frac{1}{\alpha}\Big)\quad and \quad\|\Delta f_{\alpha}\|_{L^2(\R^4)}^2=1+O\Big(\frac{1}{\alpha}\Big).$$
\end{prop}
\begin{proof}
 Let us write
  $$\|f_{\alpha}\|_{L^2(\R^4)}^{2}=I+II+III,$$
with \begin{eqnarray*}
 I&=&\int_{|x| \leq {\rm e}^{-\alpha}}|f_{\alpha}(x)|^2\; dx,\\
II&=&\int_{{\rm e}^{-\alpha}< |x|\leq 1} |f_{\alpha}(x)|^2\; dx\quad\mbox{and}\\
 III&=&\int_{|x|>1} |f_{\alpha}(x)|^2\; dx.
\end{eqnarray*}
 It is easy to see that for $\alpha$ large enough
  \begin{eqnarray*}
    I
    &\leq& 2\pi^2 \int_{0}^{{\rm e}^{-\alpha}}r^3 \left(\sqrt{\frac{\alpha}{8\pi^2}}+\frac{1}{\sqrt{32\pi^2\alpha}}\right)^2 dr\\
&\leq&\left(\frac{\alpha}{8\pi^2}+\frac{1}{32\pi^2\alpha}+\frac{1}{8\pi^2}\right)\frac{\pi^2{\rm e}^{-4\alpha}}{2}=O\Big(\frac{1}{\alpha}\Big).
     \end{eqnarray*}
Besides, by repeated integration by parts, we obtain that
\begin{eqnarray*}
 II
&=& \frac{1}{4\alpha}\Big(-\frac{\alpha^2{\rm e}^{-4\alpha}}{4}-\Int_{\rm {e}^{-\alpha}}^1 \frac{r^3}{2}\log r\; dr\Big)\\
&=&\frac{1}{4\alpha}\Big(-\frac{\alpha^2\rm e^{-4\alpha}}{4}-\frac{\alpha \rm e^{-4\alpha}}{8}+\frac{1}{32}\big(1-\rm e^{-4\alpha}\big)\Big)=O\Big(\frac{1}{\alpha}\Big).
\end{eqnarray*}
 The fact that $\eta_\alpha\in \mathcal{D}(\mathbb {R}^4)$ and $\eta_\alpha=O\Big(\Frac{1}{\sqrt{\alpha}}\Big)$ implies that $III=O\Big(\Frac{1}{\alpha}\Big)$. \\ Now, noticing that
\begin{equation*}
\nabla f_{\alpha}(x)=\left\{%
\begin{array}{ll}
\frac{-2\,x\,{\rm e}^{2\alpha}}{\sqrt{32\pi^2 \alpha}} & \hbox{si $|x|\leq {\rm e}^{-\alpha}$}, \\\\
\frac{-x}{|x|^2\sqrt{8\pi^2 \alpha}} &\hbox{si ${\rm e}^{-\alpha}< |x|\leq 1$},\\\\
\nabla\eta_{\alpha}(x) &\hbox{si $|x|> 1$},
\end{array}%
\right.
\end{equation*}
we easily get
 $$\|\nabla f_{\alpha}\|_{L^2(\R^4)}^2= \frac{{\rm e}^{-2\alpha}}{24\alpha}+\frac{1-{\rm e}^{-2\alpha}}{8\alpha}+\int_{|x|>1}|\nabla \eta_{\alpha}(x)|^2\; dx.$$
 This ensures the result knowing that $\eta_\alpha\in \mathcal{D}(\mathbb {R}^4)$ and $\|\nabla\eta_\alpha\|_{L^{\infty}}=O\Big(\Frac{1}{\sqrt{\alpha}}\Big).$\\
 Finally, since
 \begin{equation*}
\Delta f_{\alpha}(x)=\left\{%
\begin{array}{ll}
\frac{-8{\rm e}^{2\alpha}}{\sqrt{32\pi^2 \alpha}} & \hbox{if $|x|\leq {\rm e}^{-\alpha}$}, \\\\
\frac{-2}{|x|^2\sqrt{8\pi^2 \alpha}} &\hbox{if ${\rm e}^{-\alpha}< |x|\leq 1$},\\\\
\Delta\eta_{\alpha}&\hbox{if $|x|> 1$},
\end{array}%
\right.
\end{equation*}
we get
  $$ \|\Delta f_{\alpha}\|_{L^2(\R^4)}^2 dx =
   \frac{1}{\alpha}+1+\int_{|x|>1} |\Delta \eta_{\alpha}(x)|^2\; dx,$$
which ends the proof of the last assertion in view of the fact that $\eta_\alpha\in \mathcal{D}(\mathbb {R}^4)$ and \\
$|\Delta \eta_{\alpha}|=O\Big(\Frac{1}{\sqrt{\alpha}}\Big)$.
\end{proof}
In the following proposition, we recall the characterization of $H^2_{rad}(\R^4)$ which is useful in this article.
\begin{prop}\label{prop H2}
 We have
$$H^2_{rad}(\R^4)=\Big\{u\in L^2(\R_+, r^3\;dr);\quad \partial_ru,\;\partial_r^2 u,\;\frac{1}{r}\partial_ru\in L^2(\R_+, r^3\;dr)\Big\}.$$
\end{prop}
The proof of Proposition \ref{prop H2} is based on the following lemma proved in \cite{sharp}:
\begin{lem}\label{lem4}
For all $u\in H^2_{rad}(\R^4)$, we have
\begin{equation}\label{enq}
\Big\|\frac{1}{r}\partial_ru\Big\|_{L^2(\R^4)}:=\Big(2\pi^2\Int_0^\infty |u'(r)|^2r\;dr\Big)^\frac{1}{2}\leq \frac{1}{2}\|\Delta u\|_{L^2(\R^4)}.
\end{equation}
\end{lem}
\begin{proof}
By density, it suffices to consider smooth compactly supported functions. Let us then
consider $u\in \mathcal{D}_{rad}(\R^4)$. We have
\begin{eqnarray*}
\|\Delta u\|_{L^2(\R^4)}^2&=&2\pi^2\Int_0^\infty|u''(r)+\frac{3}{r}u'(r)|^2r^3\;dr\\
&=&2\pi^2\Big[\Int_0^\infty\Big(u''(r)+\frac{1}{r}u'(r)\Big)^2r^3\;dr+8\Int_0^\infty u'(r)^2r\;dr\\
&+&4\Int_0^\infty u''(r)u'(r)r^2\;dr\Big]\\
&\geq&2\pi^2\Big(8\Int_0^\infty u'(r)^2r\;dr+4\Int_0^\infty u''(r)u'(r)r^2\;dr\Big).
\end{eqnarray*}
By integration by parts, we deduce that
$$\|\Delta u\|_{L^2(\R^4)}^2\geq 8\pi^2\Int_0^\infty u'(r)^2r\;dr,$$
which achieves the proof of \eqref{enq}.
\end{proof}

It will be useful to notice, that in the radial case, we have the following estimate which implies the control of the $L^\infty$-norm far away from the origin.
\begin{prop}\label{est.rad}
Let $u\in H^1_{rad}(\R^4)$. For $r=|x|> 0$, we have
\begin{equation}\label{or1}
|u(x)|\lesssim \frac{1}{r^\frac{3}{2}}\|u\|_{L^2(\R^4)}^\frac{1}{2}\|\nabla u\|_{L^2(\R^4)}^\frac{1}{2},
\end{equation}
\end{prop}
\begin{proof}
 Let $u\in\mathcal{D}_{rad}(\R^4)$ and let us write for $r>0$,
$$u(r)^2=-2\Int_r^\infty u(s)u'(s)ds=-2\Int_r^\infty s^\frac{3}{2}u(s)s^\frac{3}{2}u'(s)\;\frac{ds}{s^3}.$$
According to Cauchy-Schwarz inequality, we obtain
\begin{eqnarray}
u(r)^2&\leq& \Big(\frac{2}{r^3} \Int_r^\infty s^3|u(s)|^2\;ds\Big)^\frac{1}{2}\Big(\frac{2}{r^3} \Int_r^\infty s^3|u'(s)|^2\;ds\Big)^\frac{1}{2}\nonumber\\
&\leq& \frac{1}{\pi^2 r^3}\|u\|_{L^2(\R^4)}\|\nabla u\|_{L^2(\R^4)},\nonumber
\end{eqnarray}
which leads to \eqref{or1} by density arguments.
\end{proof}

\noindent{\bf Acknowledgments.} {\it  We are very grateful  to Professor Hajer Bahouri and Professor Mohamed Majdoub for  interesting discussions and careful reading of the manuscript.}

\end{document}